\documentclass{amsart}
\usepackage[inner=72pt, outer= 72pt, top=108pt, bottom=102pt]{geometry}
\usepackage{amsmath}
\usepackage{amssymb,amsfonts,amsthm}
\usepackage[utf8]{inputenc}
\usepackage[english]{babel}
\usepackage[shortlabels]{enumitem}
\newtheorem{theorem}{Theorem}[section]
\newtheorem{corollary}[theorem]{Corollary}
\newtheorem{lemma}[theorem]{Lemma}

\theoremstyle{definition}
\newtheorem{definition}[theorem]{Definition}
\newtheorem{remark}[theorem]{Remark}
\newtheorem*{prR}{Property (R)}

\usepackage{amssymb}
\newcommand*{\myproofname}{Proof}

\newcommand{\N}{\mathbb N}
\newcommand{\Z}{\mathbb Z}
\renewcommand{\a}{\alpha}
\renewcommand{\b}{\beta}
\newcommand{\g}{\sigma}
\renewcommand{\c}{\gamma}
\newcommand{\Aut}{\text{{\rm Aut}}}
\newcommand{\Sh}{\text{{\rm Sh}}_G}
\newcommand{\Succ}{\text{{\rm Succ}}_G}
\newcommand{\R}{R}
\renewcommand{\l}{\ell}
\renewcommand{\d}{\mathbf{d}}
\newcommand{\SH}{S}

\begin{document}

\title{Stabilizers of consistent walks}

\author[M.~Lek\v se]{Maru\v sa Lek\v se} \address{ Faculty of Mathematics and Physics, University of Ljubljana, Slovenia; \newline Institute of Mathematics, Physics, and Mechanics, Ljubljana, Slovenia}
\email{marusa.lekse@imfm.si}

\thanks{The work was supported by Javna agencija za znanstvenoraziskovalno in inovacijsko dejavnost Republike Slovenije (ARIS), research program P1-0294.}

\maketitle

\begin{abstract}
A walk of length $n$ in a graph is consistent if there exists an automorphism of the graph that maps the initial $n-1$ vertices to the final $n-1$ vertices of the walk. In this paper we find some sufficient conditions for a consistent walk in an arc-transitive graph to have a trivial pointwise stabilizer. We show that in that case, the size of the smallest generating set of the group is bounded by the valence of the graph.
\end{abstract}

\section{Introduction}\label{s1}

For a graph $\Gamma$ and a group $G\leq \Aut(\Gamma)$, we say that a walk $(v_0, v_1, \ldots, v_n)$ of length $n \geq 1$ is \emph{$G$-consistent} if there exists an automorphism $g \in G$ (called a \emph{shunt} for the walk) such that $v_i^g = v_{i+1}$ for all $i \in \{0, \ldots, n-1\}$.

Consistent walks play a prominent role in algebraic graph theory and feature already in the seminal work of Tutte \cite{Tutte+1947, Tutte+1959} on cubic arc-transitive graphs, where he proved that the order of a vertex stabilizer $G_v$ in a $G$-arc-transitive $3$-valent graph is bounded from above by the constant $48$. In his proof, he considered the largest value of $s$ such that $G$ acts transitively on the set of all $s$-arcs of $\Gamma$ (observe that then every $s$-arc is a $G$-consistent walk). Tutte's proof consists of two parts:
\begin{enumerate}
    \item[(i)] \label{del1} show that the stabilizer of an $s$-arc is trivial,
    \item[(ii)] show that $s$ is bounded by $5$.
\end{enumerate}

Bounding the order of a vertex-stabilizer in highly symmetric graphs has become a classical theme in algebraic graph theory, see \cite{Potocnik+2007, Potocnik+2012, Spiga+2016_2, Spiga+2016, Verret+2009, Verret+2013, Weiss+1981}. 
In this paper, we will be interested in generalizations of part~(i) of Tutte's approach. In particular, we will be interested in circumstances in which we can guarantee the existence of a $G$-consistent walk with a trivial stabilizer. We will focus on the conditions imposed on the \emph{local group} $G_v^{\Gamma(v)}$, that is, the permutation group induced by the action of a vertex stabilizer on the neighborhood of that vertex. As a consequence of a technical result, Lemma \ref{dva} proved in Section \ref{s3}, the following two results are obtained: if $\Gamma$ is a $G$-arc-transitive graph and either its valence is at most $4$, or the local group $G_v^{\Gamma(v)}$ is a weakly $p$-subregular group, then there exists a $G$-consistent walk with a trivial stabilizer (see Definition \ref{def_wps} in Section \ref{s4}). In Section \ref{s6} we find examples of graphs in which no $G$-consistent walk has a trivial stabilizer.

Finally, in Section \ref{s7}, we use the techniques and results developed in this paper
to address a question raised by Marco Barbieri and Pablo Spiga in \cite{Barbieri+2024}. Indeed, \cite[Problem~4]{Barbieri+2024} asks under which additional assumptions one can guarantee existence of a function $f(d)$ such that every vertex-transitive group of automorphisms of a finite $d$-valent graph can be generated by $f(d)$ generators. As a partial answer to this question, 
we show that whenever a vertex-transitive group $G$ of automorphisms of a finite $d$-valent graph $\Gamma$ is such that there exists a $G$-consistent walk with a trivial stabilizer in $G$, the size of the smallest generating set of $G$ is bounded from above by the valence of $\Gamma$ (see Corollary \ref{CC} for a detailed statement). In particular, this implies that the assumption of having weakly $p$-subregular local group implies that $G$ is $3$-generated.

\section{Preliminaries}\label{s2}
Throughout this paper, $\Gamma$ is a finite simple graph with at least three vertices, unless otherwise stated. We say that two vertices $u,v \in V(\Gamma)$ are \emph{twin vertices}, if $\Gamma(u) = \Gamma(v)$, where $\Gamma(u)$ denotes the neighborhood of a vertex $u$ in $V(\Gamma)$. In what follows let $n$ be a positive integer and let $\a = (v_0, \ldots, v_n)$ be a $(n+1)$-tuple of vertices of $\Gamma$ such that any two consecutive vertices are adjacent in $\Gamma$; in this case we say that $\a$ is a \emph{walk} of length $n$. 
If any three
consecutive vertices in $\a$ are pairwise distinct we say that $\a$ is a \emph{$n$-arc} (or simply an \emph{arc} when $n=1$). If $v_0 = v_n$ we say that $\a$ is \emph{closed} and if in addition any two vertices in $\{v_0, \ldots, v_{n-1}\}$ are distinct, then $\a$ is a \emph{cycle}.
Note that there are no cycles of length $1$, and that every arc $(u,v)$ of length $2$ yields 
a cycle $(u,v,u)$ of length $2$; we will refer to cycles of length $2$ as \emph{trivial}. Note that contrary to the usual definition, all cycles in this paper are directed and rooted (that is, they have a designated initial vertex).
Finally if $n \geq 1$ we let $\hat \a = (v_0, \ldots, v_{n-1})$.

Let $G \leq \Aut(\Gamma)$. We say that $G$ is \emph{vertex-transitive} (or that $\Gamma$ is $G$-vertex-transitive) if it acts transitively on the set $V(\Gamma)$, that $G$ is \emph{arc-transitive} (or that $\Gamma$ is $G$-arc-transitive) if it acts transitively on the set of all arcs in $\Gamma$ and that that $G$ is \emph{n-arc-transitive} (or that $\Gamma$ is $(G,n)$-arc-transitive) if it acts transitively on the set of all $n$-arcs in $\Gamma$. 
%
Recall that a walk $\a = (v_0, \ldots, v_n)$ is $G$-consistent with a shunt $g\in G$ whenever $(v_0, \ldots, v_{n-1})^g = (v_1, \ldots, v_n)$ (when $G = \Aut(\Gamma)$ we will omit $G$ and say that $\a$ is consistent).
 The set of all shunts  of  $\a$ will be denoted by $\Sh(\a)$. 
Observe that if $\a$ is a $G$-consistent walk of length at least $2$, then also $\hat{\a}$ is $G$-consistent and $\Sh(\a) \subseteq \Sh(\hat \a)$.

If $\a$ is $G$-consistent with a shunt $g$, then $\a^x$ is $G$-consistent with a shunt $g^x$ for every $x \in G$. Note that this action preserves the set of $G$-consistent cycles in $\Gamma$.
As was pointed out by J. H. Conway in \cite{Conw_talk}, the number of $G$-orbits on the set of $G$-consistent cycles in a $G$-arc transitive graph is equal to the valence of $\Gamma$. 
The following generalization to vertex-transitive graphs was proved in \cite{Miklavic+2007}. (For further results and applications on consistent cycles, see  \cite{Antoncic+2023, Sparl2, Kutnar+2023, Miklavic+2007, Miklavic+2007_2})

\begin{theorem}[{\cite[Theorem 5.1]{Miklavic+2007}}]\label{I2}
    Let $G$ be a vertex transitive group of automorphisms of a $d$-valent graph $\Gamma$. Then $G$ has precisely $d$ orbits in its action on the set of $G$-consistent cycles of $\Gamma$. 
\end{theorem}

The following results are easy to prove and were either proved or mentioned in \cite{Miklavic+2007}. We list them here for the convenience of the reader. 

\begin{definition}
    Let $\Gamma$ be a graph and let $G \leq \Aut(\Gamma)$. For $n \geq 1$ we say that a walk $\b = (v_1, \ldots, v_n, v_{n+1})$ is a $G$-successor of a walk $\a = (v_0, \ldots, v_n)$ provided that there exists an automorphism $\g \in \Sh(\a)$ such that $v_n^g = v_{n+1}$. We denote the set of all $G$-successors of the walk $\a$ with $\Succ(\a)$.
\end{definition}

\begin{lemma}\label{odsek}
    Let $\Gamma$ be a graph, let $G \leq \Aut(\Gamma)$ and let $\a = (v_0, \ldots, v_n)$ be a $G$-consistent walk with a shunt $g \in \Sh(\a)$. Then $g G_{(v_1, \ldots, v_n)} = \Sh(\a) = G_{(v_0, \ldots, v_{n-1})} g$.
\end{lemma}

\begin{lemma}\label{X}
    Let $\Gamma$ be a graph, let $G \leq \Aut(\Gamma)$, let $\a = (v_0, \ldots, v_{n})$ be a $G$-consistent walk with a shunt $g \in \Sh(\a)$ and let $v_{n+1} = v_n^g$. Then the walk $\b = \a^g =  (v_1, \ldots, v_{n+1})$ is also $G$-consistent with a shunt $g$ and
    \begin{align*}
        \Succ({\a}) &= {\a}^{\Sh({\a})} \\
        & = \b^{G_{(v_1, \ldots, v_{n})}}.
    \end{align*}
\end{lemma}

\begin{lemma}\label{os}
    Let $\Gamma$ be a graph, let $G$ be a vertex-transitive subgroup of $\Aut(\Gamma)$, and let $\a = (v_0, \ldots, v_n)$ be a $G$-consistent walk. Then $|G_{\hat{\a}}| = |G_\a| |\Succ({\a})|$.
\end{lemma}

\begin{proof}
   Take $g \in \Sh(\a)$ and let $v_{n+1} = v_n^g$. 
   Now observe that $$|\a^{G_{\hat{\a}}}| = |\a^{g(G_{\hat{\a}})^g g^{-1}}| = |(v_1, \ldots, v_{n+1})^{G_{(v_1, \ldots, v_{n})}g^{-1}}|= |(v_1, \ldots, v_{n+1})^{G_{(v_1, \ldots, v_{n})}}| = |\Succ({\a})|,$$ by Lemma \ref{X}. Since the stabilizer of $\a$ in $G_{\hat \a}$ is $G_{\a}$, the result follows from the Orbit-Stabilizer Lemma.
\end{proof}

\section{Some conditions for a $G$-consistent walk to have a trivial stabilizer}\label{s3}

\begin{definition}
    Let $\Gamma$ be a graph and let $G \leq \Aut(\Gamma)$.
    Fix a $G$-consistent walk $\a = (v_0, \ldots, v_{n})$ of length $n \geq 2$.
    For two walks $ \beta = (b_0, \ldots, b_{n-1}),\ \gamma = (c_0, \ldots, c_{n-1})\in \hat{\a}^G$ we write $\beta \sim_{\a} \gamma$ with respect to $G$ provided that there exists a sequence of walks $\beta_0, \ldots, \beta_m \in \a^G$ such that $\hat{\beta_0} = \beta, \  \hat{\beta_m} = \gamma$ and that for every $i \in \{0, \ldots, m-1\}$ the walk $\beta_{i+1}$ is a $G$-successor of the walk  $\beta_{i}$. 
\end{definition}

We will just say that  $\beta \sim_{\a} \gamma$ and omit $G$ when $G$ is clear from context.
The next result appears without proof in \cite{Verret+2009}, we will prove it here for the convenience of the reader. The rest of the results in this section are also inspired by \cite{Verret+2009}.

\begin{lemma}\label{ekv}
    Let $G$ be a finite group that acts transitively on a set $X$. If $R$ is a transitive nonempty binary relation on the set $X$ that is preserved by $G$, it is an equivalence relation.
\end{lemma}

\begin{proof}
    Let $\Gamma$ be a not neccessarily simple digraph with vertex set $V(\Gamma) = X$ and arc set $A(\Gamma) = \{(x,y) \mid xRy\}$. Since the group $G$ preserves the relation $R$, it acts on the digraph $\Gamma$ as a group of automorphisms, this action is vertex-transitive.
    Suppose that $R$ is not a symmetric relation. Since $G$ is transitive, this implies that  for every vertex $u_0$ there exists an arc $(u_0, u_1) \in A(\Gamma)$, such that $(u_1, u_0) \notin A(\Gamma)$. Since $X$ is finite, there exists a cycle $(v_0, v_1, \ldots, v_n, v_0)$ in $\Gamma$, such that $(v_{i}, v_{i+1}) \in A(\Gamma)$ and $(v_{i+1}, v_i) \notin A(\Gamma)$ for all $i \in \{0, \ldots, n\}$, where indices are computed modulo $n$. However, by transitivity  of the  relation $R$ we see that $(v_1, v_0) \in A(\Gamma)$, which is a contradiction, implying that $R$ is a symmetric relation. Since $R \neq \emptyset$, transitivity of $G$ implies that for every $u \in X$ there exists $v \in X$ such that $uRv$. Then $vRu$, and by transitivity of $R$, also $uRu$. This shows that $R$ is reflexive and thus an equivalence relation.
\end{proof}

\begin{lemma}
    Let $\Gamma$ be a graph, let $G \leq \Aut(\Gamma)$ and let $\a = (v_0, \ldots, v_n)$ be a $G$-consistent walk. Then the relation $\sim_{\a}$ is an equivalence relation.
\end{lemma}

\begin{proof}
    Througout this proof we will use the symbol $\sim$ to denote the relation $\sim_\a$ defined on the set $\hat \a^G$.
    The relation is nonempty, since $\hat{\a} \sim \hat{\a}$.

    Let us now show that the relation $\sim$ is transitive. Let $\beta, \gamma, \delta \in \hat{\a}^G$ such that $\beta \sim \gamma$ and $\gamma \sim \delta$. Then there exist sequences of walks $\beta_0, \ldots, \beta_m$, and $\gamma_0, \ldots, \gamma_k$, where  $\beta_i, \gamma_i \in \a^G$, such that $\hat{\beta_0} = \beta, \hat{\beta_m} = \gamma$, $\hat{\gamma_0} = \gamma, \hat{\gamma_k} = \delta$, and that every walk in each of these two sequences is a $G$-successor of the previous walk in the sequence. 
    If $\beta = \gamma$, then $\beta \sim \delta$. If $\beta \neq \gamma$, then $m \geq 1$. Note that since $\hat{\gamma_0} = \hat{\beta_m}$ and since $\gamma_0, \ \beta_{m-1} \in \a^G$, the walk $\gamma_0$ is a $G$-successor of the walk $\beta_{m-1}$. 
    In the sequence of walks $\beta_0, \ldots, \beta_{m-1}, \gamma_0, \ldots, \gamma_k$, every  walk is a $G$-successor of the previous walk, thus $\beta \sim \delta$. This proves that the relation $\sim$ is transitive.

    Let us now prove that the relation $\sim$ is preserved by $G$. Let $g \in G$ and suppose that $\beta \sim \gamma$ for two walks $\beta, \gamma \in \hat{\a}^G$. Let $\beta_0, \ldots, \beta_m$ be the required walk from the definition of $\sim$. Then $\beta_0^g, \ldots, \beta_m^g$ is a sequence of walks in $\a^G$ that shows that $\beta^g \sim \gamma^g$. Hence $G$ preserves the relation $\sim$.
    Since $G$ acts transitively on its orbit $\hat \a^G$, 
    Lemma \ref{ekv} $\sim$ implies that $\sim$ is an equivalence relation.
\end{proof}

The following property of consistent walks will play a crucial role in the forthcoming analysis. Intuitively a $G$-consistent walk $\a$ has this property if every vertex of the graph can be  reached from $\a$ by a sequence of $G$-successors. Hence we shall call it \emph{reachability property} or (R) for short.

\begin{prR}
    Let $\Gamma$ be a graph and let $G$ be an arc-transitive subgroup of $\Aut(\Gamma)$. We say that a $G$-consistent walk $\a$  has Property (R) with respect to $G$, provided that for every $v \in V(\Gamma)$ there exists a sequence of walks $\b_0, \ldots, \b_m$ such that $\b_0 = \a$, the last vertex of the walk $\b_m$ is $v$, and $\b_{i+1} \in \Succ(\b_i)$ for all $i \in \{0, \ldots, m-1\}$.
\end{prR}
    
We will again omit $G$ when it is clear from context, and just say that a walk has Property (R).

\begin{lemma}\label{P1_tr}
    Let $\Gamma$ be a graph, let $G \leq \Aut(\Gamma)$ and let $\a$ be a $G$-consistent walk. Then $\a$  has Property (R) with respect to $G$ if and only if the group $\langle \Sh(\a) \rangle$ generated by all the shunts of $\a$ contained in $G$ acts transitively on $V(\Gamma)$.
\end{lemma}

\begin{proof}
    Let $\a = (v_0, \ldots, v_n)$. Suppose first that $\langle \Sh(\a) \rangle$ acts transitively on $V(\Gamma)$. Let $v \in V(\Gamma)$. If $v = v_n$, then there is nothing to prove. Henceforth assume $v \neq v_n$. By transitivity of $\langle \Sh(\a) \rangle$ there exists an automorphism $g = g_1 \cdots g_\l \in \langle \Sh(\a) \rangle$, where $g_i \in \Sh(\a)$ for all $i \in \{1, \ldots, \l\}$, such that $v_n^g =v$. Let $\b_0 = \a$ and for every $i \in \{1, \ldots, \l \}$ recursively define $h_{i} = g_{\l-i+1}^{h_1 \cdots   h_{i-1}}$ and $\b_i = \a^{h_1   \cdots   h_{i}}$. We will prove that for every $i \in  \{1, \ldots, \l \}$
    \begin{enumerate}
        \item[$($a$)$]  $h_1   \cdots   h_i = g_{\l-i+1}  \cdots   g_\l$,
        \item[$($b$)$] $h_i \in \Sh(\b_{i-1})$.
    \end{enumerate}
    In case when $i = 1$ the claims $($a$)$ and $($b$)$ follow directly from the fact that $h_1 = g_\l \in \Sh(\a) = \Sh(\b_0)$. Suppose that $i \geq 2$ and that $($a$)$ and $($b$)$ hold for all smaller $i$. Then $h_1   \cdots   h_{i-1}   h_i = h_1   \cdots   h_{i-1} g_{\l-i+1}^{h_1   \cdots   h_{i-1}} = g_{\l-i+1}   h_1   \ldots   h_{i-1} = g_{\l-i+1}  g_{\l-i+2}   \cdots   g_\l$,
    so Claim $($a$)$ is true. For any $G$-consistent walk $\gamma$ and an automorphism $f \in G$ note that $\Sh(\gamma^f) = \Sh(\gamma)^f$. 
    Since $g_{\l-i+1} \in \Sh(\a)$ and $h_1   \cdots   h_{i-1}$ maps $\a$ to $\b_{i-1}$, it follows that $h_i = g_{\l-i+1}^{h_1 \cdots h_{i-1}} \in \Sh(\b_{i-1})$, which proves Claim $($b$)$. 
    We proved that for all $i \in \{1, \ldots \l\}$ the walk $\b_i = \a^{h_1   \cdots   h_{i}} = \b_{i-1}^{h_i}$ is a 
    $G$-successor of the walk $\b_{i-1}$ and that 
    $h_1   \cdots   h_\l = g$, thus the last vertex of the walk $\b_\l$ is $v_n^{h_1   \cdots   h_\l} = v_n^{g} = v$. Therefore $\a$  has Property (R).
 
    Suppose now that the walk $\a$  has Property (R). Let $v \neq v_n$ be a vertex in $\Gamma$. We will show that there exists $h \in \langle \Sh(\a) \rangle$ such that $v_n^h = v$. Let $\b_0, \ldots, \b_\l$ be a sequence of walks such that $\b_0 = \a$, the last vertex of the walk $\b_\l$ is $v$, and $\b_{i} = \b_{i-1}^{h_i}$; $h_i \in \Sh(\b_{i-1})$ for all $i \in \{1, \ldots, \l\}$. Let $h = h_1   \cdots   h_\l$. Then $v_n^{h} = v$. 
    For all $i \in \{1, \ldots, \l\}$
    note that $\Sh(\b_i) =\Sh(\b_{i-1})^{h_i}$, and $\langle \Sh(\b_i) \rangle = \langle \Sh(\b_{i-1})^{h_i} \rangle =  \langle \Sh(\b_{i-1}) \rangle$. Thus $\langle \Sh(\b_0) \rangle = \ldots = \langle \Sh(\b_\l) \rangle$. It now follows that $h \in \langle \Sh(\b_0) \rangle = \langle \Sh(\a) \rangle$, hence the group $\langle \Sh(\a) \rangle$ is vertex-transitive.
\end{proof}

\begin{remark}\label{R1}
    Observe that by Lemma \ref{odsek}, $\langle \Sh(\a) \rangle = \langle G_{\hat{\a}},g \rangle$, where $g$ is an arbitrary element of $\Sh(\a)$. Therefore the lemma above is equivalent to saying that for a consistent cycle $\a$ with a shunt $g \in G$, the group $\langle G_{\hat{a}}, g \rangle$ is transitive on $V(\Gamma)$ if and only if $\a$  has Property (R).
\end{remark}

\begin{lemma}\label{dva}
    Let $\Gamma$ be a graph, let $G \leq \Aut(\Gamma)$ and let $\a = (v_0, \ldots, v_{n})$ be a $G$-consistent walk of length $n \geq 2$ with a shunt $g \in \Sh(\a)$. Suppose that the walk $\hat{\a}$  has Property (R). Denote $\l = |\Succ(\a)|$, $k = |\Succ(\hat{\a})|$ and $X = \langle G_{\hat{\a}}, G_{\hat{\a}^g} \rangle$. If any of the Conditions $(a)-(c)$ hold, then the walk $\a$ has Property (R) :
    \begin{enumerate}[{\rm (a)}]
        \item\label{a} there is no integer $i$, such that $\l \leq  i < k$ and $i|k$,
        \item\label{b} the group $X$ acts transitively on the set $\Succ(\hat{\a})$,
        \item\label{c} $X=G_{(\hat{\a}\cap \hat{\a}^g)}$, where $\hat{\a}\cap \hat{\a}^g = (v_1, \ldots, v_{n-1})$.  
    \end{enumerate}
\end{lemma}

\begin{proof}
    Denote $\rho = \hat{\a}$, $\tau = (v_1, \ldots, v_n)$,  and $S = \Succ(\rho)$. Note that $S = \tau^{G_{\rho \cap \tau}}$ by Lemma \ref{X}. Let $\R = \tau^{G_{\rho}}$.  Then $\R \subseteq S \subseteq \tau^G = \rho^G$. 
    
    For the sequence $\a, \a^g$ observe that $\a, \a^g \in \a^G$, $\a^g \in \Succ(\a)$, and $\hat{\a^g} = \hat\a^g = \tau$, thus by definition
    \begin{align}
        \label{DokazL1}
        \rho \sim_\a \tau.
    \end{align}

    Now take any walk in $R$. By the definition of $R$ we can write it as $\tau^{h}$ for some $h \in G_{\rho}$. Since $G$ preserves the relation $\sim_{\a}$, it follows from (\ref{DokazL1}) that $\rho = \rho^h \sim_\a \tau^{h}$. Therefore $R \subseteq [\rho]_{\sim_{\a}}$ (where $[\chi]_{\sim_\a}$ denotes the equivalence class of $\chi$ with respect to $\sim_\a$). Since $R \subseteq S$, we see that 
    \begin{align}
        \label{DokazL2}
        R \subseteq S \cap [\rho]_{\sim_{\a}}.
    \end{align}

    On the other hand,
    $G_{\rho}$ acts on $S$ and preserves the relation $\sim_\a$,
    hence $S \cap [\rho]_{\sim_{\a}}$ is a block for the action of $G_{\rho}$ on $S$. In particular
\begin{align}
        \label{DokazL3}
        |S \cap [\rho]_{\sim_{\a}}| \mathrel{|} |S|.
    \end{align}
    \medskip
    \hspace{-1px}\textbf{Claim 1:}  $S \cap [\rho]_{\sim_{\a}} = S.$
    
    Suppose first that the Condition \ref{a} holds. Recall that $|S| = k$ and that $G_{\rho}g = \Sh(\a)$ by Lemma \ref{odsek}. It follows that $|R| =  | \tau^{G_\rho}| =  |\a^{G_\rho}| =  |\a^{G_{\rho}g}| = |\a^{\Sh(\a)}| = |\Succ(\a)|= \l$. If we let $i = |S \cap [\rho]_{\sim_{\a}}|$, we see that $i | k$ by (\ref{DokazL3}), and by (\ref{DokazL2}), $\l \leq i$, which implies $i = k$, as claimed.

    Suppose now that the Condition \ref{b} holds. Note that since $G_{\rho}$ preserves the relation $\sim_{\a}$, $G_{\rho}$ fixes $[\rho]_{\sim_{\a}}$ setwise. Similarly $G_{\tau}$ fixes $[\tau]_{\sim_{\a}}$ setwise. Since $[\rho]_{\sim_{\a}}=[\tau]_{\sim_{\a}}$ by ($\ref{DokazL1}$), it now follows that $X$ preserves $[\rho]_{\sim_{\a}}$ setwise. By assumption $X$ acts transitively on $S$. Since $\rho \sim_{\a} \tau \in S$, we see that $S \subseteq [\rho]_{\sim_{\a}}$, as claimed. 

    Suppose that the Condition \ref{c} holds. The group $X=G_{(\rho\cap \tau)}$ acts transitively on the set $\Succ(\rho)$, hence \ref{c} implies \ref{b} and $S \subseteq [\rho]_{\sim_{\a}}$, as claimed. This completes the proof of Claim 1.

    We showed that  $\rho \sim_{\a} \nu$ for all $\nu \in \Succ(\rho)$. Hence for every $\beta \in \rho^G$ and every $\gamma \in \Succ(\beta)$, we see that $\gamma \sim_{\a} \beta$.

    Let  $v \in V(\Gamma)$. By assumption, the walk $\rho$  has Property (R), and hence $\tau$  has Property (R). Thus by definition there exists a sequence of walks $\b_0, \ldots, \b_m$, such that $\b_0 = \tau$, $\b_{i+1} \in \Succ(\b_i)$ for all $i \in \{0, \ldots, m-1\}$, and the last vertex of the walk $\b_m$ is $v$. We need to show that there exists a sequence of walks (that are for $1$ longer than $\b_i$) starting with $ \a$ such that every walk in the sequence is a $G$-successor of the previous walk, and that the last vertex of the last walk in the sequence is $v$. If $v = v_n$, there is nothing to prove. 
    Assume now that $v \neq v_{n}$. Then $m \geq 1$. Let $j \in \{0, \ldots, m-1\}$. By definition of the relation $\sim_\a$, there exists a sequence of walks $\gamma_{j,0}, \ldots, \gamma_{j,t_j}$, where $\gamma_{j,s} \in \a^G$ for all $s \in \{0, \ldots, t_j\}$, $\hat{\gamma_{j,0}} = \b_j$, $\hat{\gamma_{j,t_j}} = \b_{j+1}$ and every walk in the sequence is a $G$-successor of the previous walk in the sequence. Since $\b_{j+1} \in \Succ(\b_j)$, it follows that $\b_{j+1} \neq \b_j$, and therefore $t_j \geq 1$ for all $j \in \{0, \ldots, m-1\}$.

    Observe that the sequence of walks $$\a, \gamma_{0,0}, \ldots, \gamma_{0,t_1-1}, \gamma_{2,0}, \ldots, \gamma_{2,t_2-1}, \ldots, \gamma_{m-1,0}, \ldots, \gamma_{m-1,t_{m-1}-1}$$ is such that the first walk in the sequence is $\a$, the last vertex of the last walk is $v$, and every walk is a $G$-successor of the previous one, which proves that $\a$  has Property (R). 
\end{proof}

\begin{remark}
    In the above proof we proved that  Condition \ref{c} implies Condition \ref{b}. However, using the permutation group version of the Frattini argument, one can easily see that Conditions \ref{b} \ref{c} from Lemma \ref{dva} are equivalent.
\end{remark}

\begin{lemma}\label{L5}
    Let $G$ be a subgroup of $Aut(\Gamma)$ for a graph $\Gamma$. Suppose that $\a$  has Property (R) and that $G_{\a}$ fixes $\Succ(\a)$ pointwise. Then $G_\a = 1$.
\end{lemma}

\begin{proof}
    Let $g \in \Sh(\a)$. Then $\a^g \in \Succ(\a)$, thus by assumption $G_\a \subseteq G_{\a^g} = G_\a^g$. Since $|G_\a| = |G_\a^g|$, observe that $G_\a = G_\a^g$, thus $\langle \Sh(\a) \rangle$ normalizes the group $G_\a$.
    But by Theorem \ref{P1_tr} the group $\langle\Sh(\a)\rangle$ acts transitively on $V(\Gamma)$, therefore $G_\a = 1$. 
\end{proof}

\begin{corollary}\label{glavni_izrek}
    Let $\Gamma$ be a connected graph, let $G$ be an arc-transitive subgroup of $\Aut(\Gamma)$, and let $r \geq 1$.
     Let $\a = (v_0, \ldots, v_{r})$ be a $G$-consistent walk with a shunt $g \in \Sh(\a)$.  For every $n \in \{1, \ldots, r\}$ let $\a_n = (v_0, \ldots, v_n)$ and $k_n = |\Succ(\a_n)|$. 
     If $G_{\a}$ fixes all elements of the set $\Succ(\a)$ and if for every $n\in\{1,\ldots,r-1\}$ 
     at least one of the following Conditions $(a) - (c)$ holds, then $G_{\a} = 1$.
     \begin{enumerate}[{\rm (a)}]
         \item\label{aa}There is no integer $i$ such that $k_{n+1} \leq i < k_{n}$ and that $i | k_{n}$
        \item\label{bb}The group $\langle G_{\a_n}, G_{\a_n^g} \rangle$ acts transitively on the set $\Succ(\a_n)$
        \item\label{cc}$\langle G_{\a_n}, G_{\a_n^g} \rangle=G_{({\a_n}\cap{\a_n}^g)}$, where  $\a_n\cap{\a_n}^g = (v_1, \ldots, v_n)$.
     \end{enumerate}
\end{corollary}

\begin{proof}
    The group $G$ acts arc-transitively on $\Gamma$, thus for an arc $(u_0, u_1)$ every arc of the form $(u_1, u_2)$ is its $G$-successor. Since $\Gamma$ is a connected graph, the arc $(v_0, v_1)$  has Property (R). By induction it then follows from Lemma \ref{dva} that $\a$  has Property (R), hence $G_{\a} = 1$ by Lemma \ref{L5}. 
\end{proof}

The following lemma provides an equivalent condition for a $G$-successor of a walk $\a$ to be fixed by $G_\a$, we will use it in Sections \ref{s4} and \ref{s5}.

\begin{lemma}\label{alt}
    Let $\Gamma$ be a graph and let $G\leq \Aut(\Gamma)$.
     Suppose that $\a = (v_0, \ldots, v_{r})$ is a $G$-consistent walk and that $\b = (v_1, \ldots, v_{r+1})$ is a $G$-successor of $\a$. Then $G_{\a}$ fixes $\b$ if and only if $|\Succ(v_0, \ldots, v_r, v_{r+1})| = 1$.
\end{lemma}

\begin{proof}
    Since the walk $(v_0, \ldots, v_r, v_{r+1})$ is $G$-consistent, there exists an automorphism $g \in \Sh(v_0, \ldots, v_r, v_{r+1})$. By Lemma \ref{odsek}, $|\Succ(v_0, \ldots, v_r, v_{r+1})| = |(v_0, \ldots, v_r, v_{r+1})^{G_(v_0, \ldots, v_{r})g}| = |(v_0, \ldots, v_r, v_{r+1})^{G_(v_0, \ldots, v_{r})}|$. Observe that $|(v_0, \ldots, v_r, v_{r+1})^{G_(v_0, \ldots, v_{r})}| = 1$ if and only if $G_{\a}$ fixes the walk $\b$, hence $G_{\a}$ fixes the walk $\b$ if and only if $|\Succ(v_0, \ldots, v_r, v_{r+1})| = 1$.
\end{proof}

\section{Graphs with locally weakly $p$-subregular group actions}\label{s4}

In this section we will consider $G$-arc-transitive graphs where the local group $G_v^{\Gamma(v)}$, that is the subgroup of $\text{Sym}(\Gamma(v))$ induced by the action of vertex a stabilizer of $G_v$ on the set $\Gamma(v)$, belongs to a class of \emph{weakly $p$-subregular} permutation groups, first defined in \cite{Verret+2013}. As we prove in Theorem \ref{weaklypsub} one can then always find a $G$-consistent walk with a trivial $G$-stabilizer. 

\begin{definition}\label{def_wps}
    Let $p$ be a prime number. A transitive permutation group $L$ is weakly $p$-subregular on the set $\Omega$, if there exist $x, y \in \Omega$ such that $|L_x| = p$ and $x^{\overline{L}} \cup y^{\overline{L}} = \Omega$, where $\overline{L} = \langle L_x, L_y\rangle$.
\end{definition}

Some examples of weakly $p$-subregular permutation groups include the dihedral groups in their natural actions (for $p = 2$), the wreath product $C_p \wr C_2$ in its natural imprimitive action on $\{1, \ldots, p\} \times\{1,2\}$ as well as $\text{SL}(2,p)$ on $\Z_p^2 \setminus \{0\}$, see \cite{Verret+2009, Verret+2013} for further examples.

The following lemma was proved as Claim 1 in the proof of Theorem 2.2 in \cite{Verret+2013}. Note that the formulation here is equivalent to formulation in \cite{Verret+2013}, since 
 $\langle\Sh((x,v,y))\rangle = \langle G_{(x,v)}, g \rangle$, where $g \in G$ is any such automorphism that $(x, v)^g = (v,y)$. 

\begin{lemma}\label{l1wps}
    Let $\Gamma$ be a connected graph, let $G$ be an arc-transitive subgroup of $\Aut(\Gamma)$, let $v \in V(\Gamma)$ and let $L = G_v$. Suppose that there exist $x,y \in \Gamma(v)$, such that $x^{\overline{L}} \cup y^{\overline{L}} = \Gamma(v)$, where $\overline{L} = \langle L_x, L_y\rangle$. Then $\langle \Sh((x,  v, y)) \rangle$ is transitive on $V(\Gamma)$.
\end{lemma}
 
\begin{theorem}\label{weaklypsub}
let $G$ be an arc-transitive subgroup of the automorphism group of a graph $\Gamma$ of valence at least $3$ and let $v \in V(\Gamma)$. Suppose that  $G_v^{\Gamma(v)}$ is weakly $p$-subregular for a prime $p$. Then there exists a $G$-consistent walk $\a$ with $G_\a = 1$. 
\end{theorem}

\begin{proof}
    Let $L = G_v^{\Gamma(v)}$. Then by assumptions there exist vertices $x,y \in V(\Gamma)$ such that $|L_x| = |L_y| = p$ and $x^{\overline{L}} \cup y^{\overline{L}} = \Gamma(v)$, where $\overline{L} = \langle L_x, L_y\rangle$. 
    
    Let $\a$ be a maximal $G$-consistent walk that begins with the sequence $(x,v,y)$ and for which $|\Succ(\a)|>1$.
    We first need to show that such walks exist; we will show that $(x,v,y)$ satisfies the conditions. Since $G$ is arc-transitive, the walk $(x,v,y)$ is $G$-consistent. Recall that $|\Succ(x,v,y)| ={|G_{(x,v)}|} / {|G_{(x,v,y)}|}$ by Lemma \ref{os}. The definition of a weakly $p$ subregular group implies that $L_{x,y} = 1$, and thus ${|G_{(x,v)}^{\Gamma(v)}|} / {|G_{(x,v,y)}^{\Gamma(v)}|} = |{L_x}| / |{L_{x,y}}| = p$. Hence $|G_{(x,v)}| \neq |G_{(x,v,y)}|$, it now follows that  $|\Succ(x,v,y)| > 1$.
    Let $\a = (v_0, \ldots, v_n)$, where $v_0 = x, \ v_1 = v$ and $v_2 = y$, and let $g \in \Sh(\a)$.
    Denote also $v_{n+1} = v_n^g$, $\a_i = (v_0, \ldots, v_i)$ and $k_i = |\Succ(\a_i)|$ for all $i \in \{2, \ldots, n\}$. 
    We will show that for an arbitrary  $i \in \{2, \ldots, n\}$, the walk $\a_i$  has Property (R). Note that $$k_i = |\Succ(\a_i)| = |\a_i^{gG_{(v_1, \ldots, v_i)}}| = |(v_1, \ldots, v_{i+1})^{G_{(v_1, \ldots, v_i)}}| = |v_{i+1}^{G_{(v_1, \ldots, v_i)}}| = |v_{i+1}^{G_{(v_1, \ldots, v_i)}^{\Gamma(v_i)}}|.$$ The second equality follows from Lemma \ref{odsek}, the last equality holds because $G_{(v_1, \ldots,v_i)}$ acts on the set $\Gamma(v_i)$. Now since $(v_0, v_1, v_2)^{g^{i-1}} = (v_{i-1}, v_i, v_{i+1})$, observe that $G_{(v_1, \ldots, v_i)}^{\Gamma(v_i)} \leq G_{(v_{i-1}, v_i)}^{\Gamma(v_i)} = L_x^{g^{i-1}}$. Hence $|G_{(v_1, \ldots, v_i)}^{\Gamma(v_i)}|$ is either $1$ or $p$ and $k_i$ is either $1$ or $p$. By  the definition of $\a$, the parameter $k_i$ can not be $1$, so $k_i = p$ for every $i \in \{2, \ldots, n\}$.
    The walk $(x,v,y)$  has Property (R) by Lemma \ref{l1wps}. It then follows by repeated application of Lemma \ref{dva} (using Condition \ref{a}),  that $\a_i$  has Property (R)   \ 
  for all $i \in \{2, \ldots, n\}$, as claimed. In particular $\a$  has Property (R).

    Let $\b \in \Succ(\a)$, then $\b = \a^h = (v_1, \ldots, v_{n}, u)$ for some $h \in \Sh(\a)$ and $u \in V(\Gamma)$. The walk $(v_0, \ldots v_{n},u)$ is $G$-consistent with a shunt $h \in \Sh(\b)$, and begins with the vertices $x, v, y$. But since $\a$ is maximal such walk with  $|\Succ(\a)| > 1$, it follows that $|\Succ((v_0, \ldots, v_{n}, u))| = 1$. By Lemma \ref{alt}, the stabilizer $G_{\a}$ fixes $\b$ for an arbitrary $\b \in \Succ(\a)$. Hence $G_\a = 1$ by Lemma \ref{L5}.  
\end{proof}

\section{Stabilizers of $G$-consistent walks in $4$-valent graphs}\label{s5}

As the second application of Corollary \ref{glavni_izrek} we show the existence of a $G$-consistent walk with a trivial $G$-stabilizer for an arbitrary $4$-valent graph and an arc transitive $G \leq \Aut(\Gamma)$, in particular we prove the following theorem.

\begin{theorem}\label{4val}
    Let $\Gamma$ be a connected $4$-valent graph and let $G$ be an arc-transitive subgroup of $\Aut(\Gamma)$. Then there exists a $G$-consistent walk $\a$ with $G_\a = 1$.
\end{theorem}

\begin{proof}
Let $v \in V(\Gamma)$ and $u \in \Gamma(v)$. Then $G_{(u,v)}$ acts on the set $\Omega = \Gamma(v) \setminus \{u\}$, the local group $H = G_{(u,v)}^{\Omega}$ is either trivial, $C_2$, $C_3$ or $S_3$. 

If $H$ is the trivial group, we take $(u,v)$ as $\a$. The walk $\a$  has Property (R) , since $\Gamma$ is connected and $G$ acts transitively on the set of arcs of $\Gamma$. The stabilizer $G_\a$ fixes the set $\Succ(\a)$ pointwise, hence $G_\a = 1$ by Corollary \ref{glavni_izrek}. 

If $H$ is permutation isomorphic to $C_2$, then $G_v^{\Gamma(v)}$ is permutation isomorphic to $D_4$, so it is weakly $p$-subregular on $\Gamma(v)$. A $G$-consistent walk $\a$ with $G_\a = 1$ then exists by Theorem \ref{weaklypsub}.

In case when $H$ is permutation isomorphic to $C_3$ or $S_3$, $G$ acts $2$-arc transitively on $\Gamma$, hence an arbitrary $2$-arc $(u, v, w)$ is a $G$-consistent walk. Let $\a = (v_0, \ldots, v_r)$ be the longest $G$-consistent walk such that $|\Succ(\a)| > 1$, $v_0 = u, v_1 = v$ and $v_2 = w$. We will show that then $\a$ satisfies the conditions from Corollary \ref{glavni_izrek}.
For all $z \in \Gamma(v)$ the walk $(v, z)$ is a $G$-successor of $(u,v)$, since $G$ acts arc-transitively on $\Gamma$. Hence $|\Succ(v,u)| = 4$. Similarly the $G$-successors of the walk $(u,v,w)$ are the walks $(v,w,z)$ for $z \in \Gamma(u) \setminus \{v\}$, since $G$ acts $2$-arc-transitively on $\Gamma$. Hence $|\Succ(v,u,w)| = 3$. 

Let $\a_i = (v_0, \ldots, v_i)$ for all $i = \{1, \ldots, r \}$ . For all $i,j  < r$, $i\leq j$, the walks $\a_{i} $ and $\a_{j} $ are $G$-conistent and $|\Succ(\a_{j} )| \leq |\Succ(\a_i )|$. Since there is no divisor $k$ of $3$ or $\l$ of $2$ such that $1 < k < 3$ or $1 < \l <2$, each $\a_i$ is $G$-consistent and satisfies Condition \ref{aa} in Corollary \ref{glavni_izrek}. By the definition of $\a$ we know that $|\Succ(\b)| = 1$ for any $G$-consistent walk $\b = (v_0, \ldots, v_r, x)$ and $x \in V(\Gamma)$. Hence  the stabilizer $G_\a$ fixes all elements of the set $\Succ(\a)$ by Lemma \ref{alt}. It follows that $G_\a = 1$  By corollary \ref{glavni_izrek}.
\end{proof}

\begin{remark}
    Theorems \ref{I2}, \ref{4val} and the Orbit-Stabilizer Lemma imply that for a connected $4$-valent graph $\Gamma$, an arc-transitive $G \leq \Aut(\Gamma)$ and $v \in V(\Gamma)$ with $|G_v| = 3^s 2^t$, there exist $4$ orbits consistent cycles. One of the orbits contains all trivial cycles, which all have $3^s 2^{t-2}$ shunts. For two of the remaining orbits, every cycle has a unique shunt, while any cycle in the last orbit has $2^{t-2}$ distinct shunts.
\end{remark}

\section{Examples of graphs for which there exists no $G$-consistent walk with a trivial stabilizer}\label{s6}

For a non-prime integer $d \geq 6$ there always exist connected arc-transitive $d$-valent graphs in which no consistent cycle has a trivial stabilizer. If $d = a \cdot b$ with $a,b \neq 1$ and $a \geq 3$, we can start with any connected arc-transitive $a$-valent graph $\Gamma$ that has no twin vertices and is not isomorphic to $\text{Cay}(\Z_n, S)$ with $\pm 1 \in S$, for example the hypercube $Q_a$ (see \cite{Kovacs+2004} for the classification of arc-transitive and  \cite{Alspach+1996} for the classification of $2$-arc transitive circulants). As we prove below, the graph $\Gamma[\overline{K_b}]$, where $\overline{K_b}$ denotes the empty graph on $b$ vertices, is then an arc-transitive $d$-valent graph in which no cycle has a trivial stabilizer.

As noted in \cite{Miklavic+2007}, a $G$-consistent walk $\a = (v_0, \ldots, v_n) $ and a shunt $g \in \Sh(\a)$ induce a unique $G$-consistent cycle $\b$. If no vertices in $\a$ repeat, then $\b = (v_0, \ldots, v_n, v_n^g, v_n^{g^2}, \ldots v_n^{g^i})$ where $i$ is the first integer for which $v_n^{g^i} = v$ for some vertex $v$ in $\a$, clearly $v$ must then be $v_0$.
If $v_k$ is the first vertex in $\a$ that is the same as one of the previous vertices, then $v_k = v_0$ and $v_x = v_{x \mod k}$ for all $x \in \{0, \ldots, n\}$. In this case the induced cycle $\b$ is $(v_0, \ldots, v_k)$.

\begin{lemma}
    Let $\Gamma$ be a connected vertex-transitive graph that has no twin vertices and is not isomorphic to $\text{Cay}(\Z_n, S)$ with $\pm 1 \in S$ for any $n \in \N$. Let $m \geq 2$. Then $\tilde \Gamma = \Gamma[\overline{K_m}]$ contains no consistent walk with a trivial stabilizer. 
\end{lemma}

\begin{proof}
   Every consistent walk induces a consistent cycle, whose stabilizer is a subgroup of the stabilizer of the consistent walk. Hence it is enough to show that $\tilde \Gamma$ contains no consistent cycle with a trivial stabilizer.
    If a consistent cycle in $\Gamma$ contains all the vertices of $\Gamma$, then its shunt generates a cyclic vertex-regular subgroup of $\Aut(\Gamma)$, which contradicts the assumption that $\Gamma$ is not isomorphic to $\text{Cay}(\Z_n, S)$ with $\pm 1 \in S$ for some $n \in \N$. 
    Hence there exists no consistent cycle in $\Gamma$ that constains all the vertices of $\Gamma$.

     Suppose that $\Gamma$ is a $d$-valent graph. By Theorem \ref{I2}, there exist $d$ orbits of consistent cycles in $\Gamma$. We denote their representatives by $\a_1, \ldots, \a_d$ (note that one of them is the trivial cycle that only contains two distinct vertices). For every $i \in \{1, \ldots, d\}$ we can build $m$ consistent cycles $\b_{i,1}, \b_{i,2} \ldots \b_{i,m}$ in $\tilde \Gamma$. Let $\a_i = (v_0, v_1, \ldots, v_{k-1}, v_k, v_0)$, let $g_i$ be a shunt for $\a_i$ and let $V(\overline{K_m}) = \{0, \ldots, m-1\}$. Then, define 
     \begin{align*}
         \b_{i,1} &= ((v_0, 0), (v_1,0), \ldots, (v_{k-1},0) (v_k,0), (v_0,0)), \\
         \b_{i,2} &= ((v_0, 0), (v_1,0), \ldots, (v_{k-1},0) (v_k,0), (v_0,1), (v_1,1), \ldots, (v_{k-1},1) (v_k,0)), \\
         &\vdots \\
         \b_{i,m} &= ((v_0, 0), (v_1,0), \ldots, (v_{k-1},0) (v_k,0), (v_0,1), (v_1,1), \ldots, (v_{k-1},1) (v_k,2), \ldots \\ & \ldots, (v_0,m-1), (v_1,m-1), \ldots, (v_{k-1},m-1) (v_k,0)).
     \end{align*}
    Note that all the sequences $\b_{i,1}, \b_{i,2}, \ldots, \b_{i,m}$ are cycles in $\tilde{\Gamma}$. They are consistent, since $S_m \wr \Aut(\Gamma)$ (where $S_m$ acts naturally on $\{0,1, \ldots, n-1 \}$) is a subgroup 
    of the automorphism group of $\tilde \Gamma$, and $(((0 \ 1 \ \ldots \ j), \text{id}, \ldots, \text{id}), g_i) \in S_m \wr \Aut(\Gamma)$ is a shunt for $\b_{i,j}$.
    
    We claim that the cycles $\b_{i,j}$ are representatives of different $\Aut(\tilde \Gamma)$-orbits of consistent cycles in $\tilde \Gamma$. Observe that for all $i, j_1, j_2$, the cycles $\b_{i,j_1}$ and $\b_{i, j_2}$ have different length, hence they are included in different $\Aut(\tilde \Gamma)$-orbit of consistent cycles. 

    We are left with the case $i_1 \neq i_2$.
    Suppose there exist $j_1, j_2 \in \{1, \ldots, m\}$ and $g \in \Aut(\tilde{\Gamma})$ such that $\b_{i_1,j_1}^g = \b_{i_2,j_2}$. Since $\Gamma$ has no twin vertices, the sets $\{v_i\} \times \{0, \ldots, m-1\}$ for $v_i \in V(\Gamma)$ are a block system for the action of $\Aut(\tilde \Gamma)$. Thus $g$ determines an automorphism $g' \in \Aut(\Gamma)$, where $g'(v_i)$  is the first component of $g(v_i, 0)$. But $\a_{i_1}^{g'} = \a_{i_2}$, which is a contradiction with the fact that  $\a_{i_1}$ and $\a_{i_2}$ are representatives of different $\Aut(\Gamma)$-orbits of consistent cycles in $\Gamma$. This concludes the proof of the claim.

    For every $\b_{i,j}$, let $v_i \in V(\Gamma)$ be a vertex that is not included in $\a_i$.
    Then every permutation $k$ that exchanges $(v_i, 0)$ and $(v_i, 1)$
    while fixing all other vertices of $\tilde{\Gamma}$ is an automorphism of $\tilde \Gamma$ that fixes $\b_{i,j}$. Hence $\b_{i,j}$ does not have a trivial stabilizer.
    
    Since $\tilde\Gamma$ is a vertex-transitive $dm$-valent graph, there exists $dm$ $\Aut(\tilde \Gamma)$-orbits of consistent cycles by Theorem \ref{I2}. There are $dm$ many cycles $\b_{i,j}$, it follows that must include one representative from each $\Aut(\tilde \Gamma)$-orbit of consistent cycles. Therefore no consistent cycle in $\tilde \Gamma$, and hence no consistent walk in $\tilde \Gamma$, has a trivial stabilizer.
\end{proof}

\section{Consistent cycles and group generators}\label{s7}

Let $\Gamma$ be a connected $d$-valent graph and let $G$ be a vertex-transitive subgroup of $\Aut(\Gamma)$.
Recall, that by a result of Conway~\cite{Conw_talk,Miklavic+2007} there exist $d$ $G$-orbits of $G$-consistent cycles in $\Gamma$. Moreover, as proved in \cite[Lemma~3.5]{Miklavic+2007_2}, one can choose one representative from each of the orbits so that the representatives of any two orbits share as many initial vertices as is possible for any two cycles from those orbits. Let us make this more precise.

\begin{definition}[{\cite{Miklavic+2007_2}}]
Let $\a=(a_0, \ldots, a_k)$ and $\b=(b_0, \ldots, b_\l)$ be two walks. Then their overlap $m(\a,\b) = -1$ if $a_0 \neq b_0$, and, if $a_0 = b_0$, 
\[m(\a, \b) = \text{max}\{t \mid a_i = b_i \hbox{ for } i = 0, 1, 2,...,t\}.\] Similarly, if $A=\a^G$ and $B=\b^G$ are two distinct $G$-orbits of $G$-consistent walks, then
\[m(A,B) = \max\{m(\a,\b) : \a \in A, \b \in B\} .\]
\end{definition}

\begin{lemma}[{\cite[Lemma 3.5]{Miklavic+2007_2}}]
    If $G \leq \Aut(\Gamma)$, then there exists a complete set of representatives $\b_1, \ldots, \b_d$ of the $G$-orbits of $G$-consistent cycles such that $m(\b_i,\b_j) = m(\b_i^G,\b_j^G)$ for every pair $i,j$, $i\not = j$.
\end{lemma}

Let $\mathcal{C} = \{\b_1, \ldots, \b_d\}$ be such a set of representatives of $G$-orbits of $G$-conistent cycles and let $g_1, \ldots, g_d \in G$ be one shunt for each of $\b_1, \ldots, \b_d$ respectively. 
Let $\SH = \langle g_1, \ldots, g_d \rangle$. We will prove that if there exists a $G$-consistent cycle $\a$ such that $G_\a = 1$, then $\SH = G$. In particular $\d(G) \leq \text{val}(G)$, where $\d(G)$ denotes the size of the smallest generating set of $G$.

\begin{theorem}
    Let $\Gamma$ be a connected graph, let $G$ be a vertex-transitive subgroup of $\Aut(\Gamma)$ and let $\tau$ be a $G$-consistent walk. Suppose that $H \leq G$ is such that $G_\tau \leq H$ and that $\SH \subseteq H$. Then $H = G$.
\end{theorem}

\begin{proof}
    Let $g \in \Sh(\tau)$ and 
    let $\a = (v_0, \ldots, v_n)$ be the $G$-consistent cycle induced by $g$ and $\tau$. We can assume that $\a = \b_1$ and $g = g_1$. For $i \in \{1, \ldots, n\}$ let $\a_i = (v_0, \ldots, v_i)$. 
   Since $G$ acts on the set of all $G$-consistent walks of length $i-1$, note that for every $i \in \{2, \ldots, n\}$ (see Lemma \ref{odsek}):
    \begin{align}\label{l_odsek}
        G_{\a_{i-1}} = \Sh(\a_i)g^{-1}. 
    \end{align}
We will inductively prove that for every $i \in \{1, \ldots, n\}$:
    \begin{enumerate}[{\rm (i)}]
        \item\label{i} $G_{\a_{i-1}} \leq H$,
        \item\label{ii} $\Sh(\gamma) \subseteq H$ for every $G$-consistent cycle $\gamma$ for which $m(\gamma, \a) \geq i$.
    \end{enumerate}
First let $i = n$. Then the walks $\a = \a_n$ and $\a_{n-1}$ contain the same vertices, hence $G_{\a_{n-1}} = G_{\a_{n}} \leq G_\tau \leq H$. The only $G$-consistent cycle that has overlap at least $n$ with $\a$ is $\gamma = \a$. Since  $\Sh(\a) = G_{\a_{n-1}} g $ by (\ref{l_odsek}), and since $g \in H$, it follows that $\Sh(\gamma) \subseteq H$.

Suppose now that the claim is true for $i+1$, we will prove that it is then true for $i$. 
Let $\gamma$ be a $G$-consistent cycle such that $m(\gamma, \a) \geq i$. There exist
$\b \in \mathcal{C}$ and $h \in G$ such that $\b^h = \c$. Let $g_\b \in \SH \subseteq H$ be a shunt for $\b$. 
Then $s = (g_\b)^h \in \Sh(\c)$. Since $\b$ has the largest possible overlap with $\a$ of any 
cycle in its orbit, it follows that $m(\b, \a) \geq m(\gamma, \a) \geq i$, hence $h \in G_{\a_i} \leq H$. Therefore $s \in H$. 
Let $\gamma_{i+1} = (c_0, \ldots, c_i, c_{i+1})$ be the subwalk of $\gamma$ that contains the first $i+2$ vertices of $\gamma$. 
Then $\Sh(\gamma) \subseteq \Sh(\gamma_{i+1})$. Since the overlap of $\gamma$ and $\a$ is at least $i$, it follows that $v_k = c_k$ for all $i \in \{0, \ldots, i\}$. By Lemma \ref{odsek}, $\Sh(\gamma_{i+1}) = G_{\a_i}s$, hence $\Sh(\gamma) \subseteq H$, which proves part \ref{ii} of the claim.

To prove part \ref{i} of the claim note that 
$$\Sh(\a_i) = \bigcup_{\substack {\text{$\delta$; $\delta$  is a $G$-consistent cycle} \\ \text{and } \ m(\delta,\a) \geq i}} \Sh(\delta),$$
since for every element of $t \in \Sh(\a_i)$ there exists a $G$-consistent cycle with a shunt $t$ that begins with $\a_i$.
Hence by part \ref{ii} it follows that $\Sh(\a_i) \leq H$. 
Since $G_{\a_{i-1}} = \Sh(\a_i)g^{-1}$ by (\ref{l_odsek}), it follows that $G_{\a_{i-1}} \leq H$, which proves the claim.

We proved that $G_{v_0} \leq H$. Let $u \in \Gamma(v_0)$, we will show that there exists $s_u \in H$ such that $v_0^{s_u} = u$. Since $G$ is vertex-transitive, there exists an automorphism that maps $v_0$ to $u$. This induces a $G$-consistent cycle $\delta$ that begins with the vertices $v_0, u$.
There exist $\b \in \mathcal{C}$ and $h_u \in G_v$ such that $\b^{h_u} = \delta$. Let $g_\b \in \SH$ be a shunt for $\b$. Then $s_u = (g_\b)^{h_u}$ is a shunt for $\delta$. Since $h_u \in G_{v_0} \leq H$ and $g_\b \in \SH \subseteq H$, it follows that $s_u \in H$. Hence there exists an element of $H$ that maps $v_0$ to $u$.
By a connectedness argument it follows that $H$ is vertex-transitive. On the other hand, $G_{v_0} \leq H$, hence $G = H$ by the Frattini argument.
\end{proof}

\begin{corollary}\label{CC}
    Let $G$ be a vertex-transitive group of automorphisms of a connected $d$-valent graph $\Gamma$, and let $\a$ be a $G$-consistent walk. Then
    \[ \d(G) \leq d + \d(G_{\a}) .\] 
    In particular, if $\Gamma$ contains a $G$-consistent cycle $\a$ such that $G_{\a} = 1$, then $\d(G) \leq d$.
\end{corollary}

To conclude, Corollary~\ref{CC} can be refined for specific local actions to give a partial solution to \cite[Problem~4]{Barbieri+2024}. Indeed, in Theorem~\ref{weaklypsub}, we have proved that the stabilizer of a consistent cycle is trivial whenever the local group is weakly $p$-subregular.

\begin{corollary}\label{CCC}
    Let $G$ be a vertex transitive group of automorphisms of a connected $d$-valent graph $\Gamma$. Suppose that the local group is weakly $p$-subregular for some primes $p$. Then $\d(G) \leq d$.
\end{corollary}

\section*{Acknowledgement}

This paper was written as a part of a doctoral thesis at the University of Ljubljana. The author is grateful to her supervisor Primož Potočnik for all of his help and guidance.

\end{document}